\documentclass[12pt]{amsart}

\usepackage{amsmath,amssymb,amsthm}
\usepackage{mathtools}
\usepackage[margin=1in]{geometry}
\usepackage{enumitem}
\numberwithin{equation}{section}
\usepackage[colorlinks=true,linkcolor=blue,citecolor=blue,urlcolor=blue]{hyperref}
\usepackage[backend=bibtex,style=alphabetic,maxbibnames=99]{biblatex}
\theoremstyle{plain}
\newtheorem{theorem}{Theorem}[section]
\newtheorem{proposition}[theorem]{Proposition}
\newtheorem{lemma}[theorem]{Lemma}
\newtheorem{corollary}[theorem]{Corollary}
\theoremstyle{definition}
\newtheorem{definition}[theorem]{Definition}
\newtheorem{remark}[theorem]{Remark}
\newtheorem{assumption}[theorem]{Assumption}

\newcommand{\R}{\mathbb{R}}

\newcommand{\dd}{\,\mathrm{d}}
\newcommand{\eps}{\varepsilon}
\renewcommand{\d}{\partial}
\newcommand{\Lop}{\mathcal{L}}

\title[A convexity-type invariant for critical coagulation--fragmentation]
{A convexity-type invariant for the
critical\\ coagulation--fragmentation Hamilton--Jacobi equation}

\author{Truong-Son P. Van}
\address{Ho Chi Minh City University of Technology\\ Vietnam National
University Ho Chi Minh City}
\email{truongson.vanp@gmail.com}

\date{\today}

\begin{document}

\begin{abstract}
	We study the critical coagulation--fragmentation equation with multiplicative
	coagulation kernel $a(s,\hat s)=s\hat s$ and constant fragmentation kernel
	$b(s,\hat s)=1$. Under the Bernstein transform, mass-conserving solutions
	correspond to solutions of a singular Hamilton--Jacobi equation studied by
	Tran and Van (\emph{Comm.\ Pure Appl.\ Math.}\ \textbf{75} (2022), no.\ 6,
	1292--1331).
	Through this correspondence they proved that mass-conserving
	solutions are unique on the full critical range $0<m\le1$, but could establish
	their existence only for $0<m<\tfrac12$. We identify
	a one-sided,
	convexity-type invariant that holds for
	Bernstein-transform data and is propagated by their viscous scheme as a
	genuine maximum-principle bound. We call it the \emph{half-slope
		invariant}. It sharpens the curvature barrier and thereby
	extends mass-conserving existence to the entire critical range
	$0<m\le1$. Hence $m=1$ is the critical mass, confirming the threshold
	predicted by Vigil and Ziff (\emph{J.\ Colloid Interface Sci.}\ \textbf{133}
	(1989), no.\ 1, 257--264). The same invariant appears in the radial
	partial-mass formulation of the two-dimensional Keller--Segel equation, whose
	critical mass is $8\pi$.
\end{abstract}

\maketitle

\section{Introduction}

The coagulation--fragmentation equation describes the evolution of the size
distribution of a population of clusters subject to binary coalescence and
binary break-up.
For  density $c(s,t)\ge 0$ of clusters of (continuous) size $s \geq 0$ at time $t\geq 0$,
it reads
\begin{equation}\label{eq:CF}
	\d_t c(s,t) = Q_c(c)(s,t)+Q_f(c)(s,t),
\end{equation}
where, for nonnegative symmetric coagulation and fragmentation kernels $a$ and
$b$, the coagulation and fragmentation operators are
\begin{align}
	Q_c(c)(s,t) & =\frac12\int_0^s a(y,s-y)\,c(y,t)\,c(s-y,t)\,\dd y
	-c(s,t)\int_0^\infty a(s,y)\,c(y,t)\,\dd y,\label{eq:Qc}         \\
	Q_f(c)(s,t) & =-\frac12\,c(s,t)\int_0^s b(s-y,y)\,\dd y
	+\int_0^\infty b(s,y)\,c(y+s,t)\,\dd y.\label{eq:Qf}
\end{align}

The coagulation mechanism dates back to
Smoluchowski~\cite{Smoluchowski}. The mathematical theory of the full
equation, covering existence, uniqueness, and the long-time behaviour of
solutions, has since been developed extensively. We refer the reader to
Ball--Carr~\cite{BallCarr} and the monograph of
Banasiak--Lamb--Lauren\c{c}ot~\cite{BLL} for a comprehensive account.

A defining feature of these equations is that solutions need not conserve mass.
When coagulation dominates, mass can be lost to clusters of infinite size in
finite time, a phenomenon known as \emph{gelation}, while strong fragmentation
can produce clusters
of size zero~\cite{EMP,ELMP}. For the borderline multiplicative coagulation
kernel $a(s,\hat s)=s\hat s$ and constant fragmentation kernel $b\equiv1$,
Escobedo--Lauren\c{c}ot--Mischler--Perthame~\cite{ELMP} singled out a critical
regime in which the occurrence of mass conservation or gelation is governed by
the \emph{size of the initial mass} rather than by the kernels. A balance
computation of Vigil--Ziff~\cite{VigilZiff} for the zeroth moment predicts the
critical mass to be $m=1$: globally mass-conserving solutions are expected for
initial mass $m\le1$, and gelation for $m>1$. Proving this dichotomy, and in
particular the existence of mass-conserving solutions throughout the critical
range $0<m\le1$, has remained open. Existence has been available only for small
mass: Lauren\c{c}ot~\cite{Laurencot} obtained mass-conserving solutions for
$0<m<\tfrac{1}{4\log2}$, subsequently improved to $0<m<\tfrac12$ in~\cite{TV}.
For arbitrary $m>0$ with finite second moment, Tran and the author constructed
in~\cite{TVlocal} a unique local-in-time mass-conserving solution, so on the
critical range the question is whether the local solution extends globally.

A fruitful approach to this borderline problem goes back to Menon and
Pego~\cite{MenonPego}, who introduced the \emph{Bernstein transform}, a
desingularized Laplace transform, to analyse Smoluchowski's coagulation
equation. The idea is to convert the nonlocal coagulation--fragmentation
equation into a
\emph{local} equation for the transform. Under mass conservation the Bernstein
transform $F(x,t)=\int_0^\infty(1-e^{-sx})c(s,t)\dd s$ satisfies a singular
Hamilton--Jacobi equation, \eqref{eq:HJ} below. In~\cite{TV}, Tran
and the author
developed a theory of viscosity solutions for this equation, which lies outside
the classical Crandall--Lions framework because of the singular zeroth-order
term $F/x$. They used it to prove uniqueness of mass-conserving solutions on the
entire critical range $0<m\le1$, existence on the subrange $0<m<\tfrac12$, and
non-existence for $m>1$.
A version of this non-existence result appeared earlier in~\cite{BLL}, via a different method.

The existence proof rests on a one-sided
curvature bound
$xF_{xx}\ge-1$ for the viscosity solution.
It is precisely the derivation of
this bound that has been confined to $m<\tfrac12$.

The purpose of this work is to remove that confinement. We identify a
single a priori estimate, the one-sided, convexity-type invariant
\(W := 2M - xM_x \ge 0\) for \(M := mx - F\) (see~\eqref{eq:invariant}).
We call this the \emph{half-slope invariant}.
It holds
automatically for Bernstein-transform initial data and is propagated
by the viscous approximations of \cite{TV}. It sharpens the estimates
behind the curvature barrier, so that \(xF_{xx} > -1\) now holds on
the entire range \(0 < m \le 1\).
This, in turn, is the key to
proving that \(F\) is Bernstein (made precise in Section~\ref{sec:Bernstein}). Hence we finally confirm the
predicted critical mass \(m = 1\).

\subsection{The critical case and its Bernstein transform}
\label{sec:Bernstein}
Throughout we take the multiplicative coagulation kernel and constant
fragmentation kernel
\begin{equation}\label{eq:kernels}
	a(s,\hat s)=s\hat s,\qquad b(s,\hat s)=1 .
\end{equation}

We recall from \cite{TV} the notion of solution we work with.

\begin{definition}[Weak solution in the measure sense]\label{def:measure}
	For each $t\ge0$, let $c_t(\dd s)$ be a positive Radon measure on
	$(0,\infty)$. We say that $(c_t)_{t\ge0}$ is a \emph{weak solution in the
		measure sense} of \eqref{eq:CF} with the kernels \eqref{eq:kernels} if, for
	every test function $\phi\in BC([0,\infty))\cap\mathrm{Lip}([0,\infty))$
	with $\phi(0)=0$,
	\begin{multline*}
		\frac{\dd}{\dd t}\int_0^\infty\phi(s)\,c_t(\dd s)
		=\frac12\int_0^\infty\!\!\int_0^\infty
		\bigl(\phi(s+\hat s)-\phi(s)-\phi(\hat s)\bigr)\,
		s\hat s\,c_t(\dd s)\,c_t(\dd\hat s)\\
		-\frac12\int_0^\infty\!\!\int_0^s
		\bigl(\phi(s)-\phi(\hat s)-\phi(s-\hat s)\bigr)\dd\hat s\,c_t(\dd s) .
	\end{multline*}
	The solution is \emph{mass-conserving} if
	$m_1(t)=\int_0^\infty s\,c_t(\dd s)=m_1(0)$ for all $t\ge0$.
\end{definition}

Vigil and Ziff~\cite{VigilZiff} predicted that the critical mass
should be $m=1$
via the following simple calculation.
Integrating \eqref{eq:CF}
against $\phi\equiv 1$ gives, for the zeroth moment $m_0(t)=\int_0^\infty
	c(s,t)\dd s$,
\begin{equation}\label{eq:zeroth}
	\frac{\dd}{\dd t}m_0(t)=\tfrac12\,m_1(t)\bigl(1-m_1(t)\bigr).
\end{equation}
If $m_1\equiv m>1$, then $m_0$ becomes negative in finite time, signalling
gelation. For $0<m\le1$ the zeroth moment stays nonnegative.
This, however, is purely formal because $\phi \equiv 1$ does not belong to the class of test
functions in Definition~\ref{def:measure}.

Our main result is the following.

\begin{theorem}\label{thm:main}
	Let $c_0=c(\cdot,0)$ be a nonnegative measure on $(0,\infty)$ with
	$m_1(0)=m\in(0,1]$ and bounded zeroth and second moments, that is,
	\[
		m_0(0)=\int_0^\infty c_0(s)\dd s<\infty
		\qquad\text{and}\qquad
		m_2(0)=\int_0^\infty s^2c_0(s)\dd s<\infty .
	\]
	Then the coagulation--fragmentation equation~\eqref{eq:CF} has a unique
	mass-conserving weak solution in the measure sense of
	Definition~\ref{def:measure}.
	For $m>1$, by contrast,
	\eqref{eq:CF} has no global mass-conserving solution.
	Hence $m=1$ is the
	critical mass.
\end{theorem}

Following~\cite{TV}, we apply the \emph{Bernstein transform}
\begin{equation}\label{eq:Bernstein}
	F(x,t)\stackrel{\text{def}}{=}\int_0^\infty\bigl(1-e^{-sx}\bigr)\,c(s,t)\dd s,
	\qquad (x,t)\in[0,\infty)^2 .
\end{equation}
For each fixed $x\ge0$, the function $s\mapsto1-e^{-sx}$ is bounded and
Lipschitz on $[0,\infty)$ and vanishes at $s=0$, so it is an admissible test
function in Definition~\ref{def:measure}.
Writing $m_1(t)=\int_0^\infty s\,c(s,t)\dd s$ for the total mass and assuming
mass conservation $m_1(t)\equiv m$, the transform $F$ solves the singular
Hamilton--Jacobi equation
\begin{equation}\label{eq:HJ}
	\begin{gathered}
		F_t+\tfrac12\,(F_x-m)(F_x-m-1)+\frac{F}{x}-m=0
		\quad\text{in }(0,\infty)^2,\\
		0\le F\le mx ,\qquad F(x,0)=F_0(x).
	\end{gathered}
\end{equation}
The constraint $0\le F\le mx$ is intrinsic to the transform: $c\ge0$ gives
$F\ge0$, while $1-e^{-sx}\le sx$ gives $F\le mx$.
The Hamiltonian
\begin{equation*}
	H(p,z,x)=\tfrac12(p-m)(p-m-1)+\frac{z}{x} -m
\end{equation*}
is monotone but not Lipschitz in $z$ as
$\d_z H=1/x\to+\infty$ as $x\to0^+$, so \eqref{eq:HJ} lies outside the classical
Crandall--Lions theory.
In~\cite{TV}, the authors extended the classical theory of viscosity
solutions to study~\eqref{eq:HJ}.

The real difficulty is the following.
To produce a mass-conserving solution $c$ of~\eqref{eq:CF},
one must show that the viscosity solution $F$ of~\eqref{eq:HJ}
is Bernstein in $x$ for each $t>0$, .i.e,
\[
	F\in C^\infty((0,\infty)^2)\cap C^1([0,\infty)^2)
\]
and
\[
	(-1)^{n+1}\,\d_x^n F\ge0\quad\text{on }(0,\infty)^2\quad\text{for
		every }n\ge1 \, ,
\]
with the boundary derivative $F_x (0,t) = m$ so that $F$ is the transform of a measure of mass $m$.

\subsection{The known results}
We recall the relevant results of~\cite{TV}. On the full critical range there is uniqueness:
\begin{quote}
	\emph{If $0<m\le1$ and $F_0$ is Lipschitz, sublinear, $0\le F_0\le mx$, then
		\eqref{eq:HJ} has a unique Lipschitz sublinear viscosity solution},
	\emph{whence \eqref{eq:CF} has at most one mass-conserving
		solution}.
\end{quote}
For $m>1$ there is non-existence: \eqref{eq:HJ} admits no sublinear
$C^1$ solution, so \eqref{eq:CF} has no global mass-conserving
solution.
The existence side, however, is established only on a strict
subrange:
\begin{quote}
	\emph{If $F_0$ is the Bernstein transform of $c_0$ with
		$m_1(0)=m\in(0,\tfrac12)$
		and bounded zeroth and second moments, then the viscosity solution $F$ of~\eqref{eq:HJ}
		is Bernstein,
		whence \eqref{eq:CF} has a
		mass-conserving
		weak (measure) solution}.
\end{quote}
Thus the window $\tfrac12\le m\le1$ is open.
The purpose of this work is to
isolate a single a~priori estimate which, once established, removes the
artificial factor $\tfrac12$ and brings existence up to the critical mass.

On the
discrete side, Jang and Tran~\cite{JangTran} state the critical-mass
conjecture in the form used here and prove existence of mass-conserving
solutions exactly on the range $0<m<\tfrac12$, so the discrete analogue of the
window $\tfrac12\le m\le1$ remains open as of this writing.

Finally, we mention that the large time behaviour of solutions
of~\eqref{eq:HJ} was studied by Mitake, Tran, and the author~\cite{MTV}, with
complete characterizations of the stationary solutions and optimal conditions
for convergence.

\subsection{Standing hypotheses on the initial datum}
Throughout we adopt conditions in~\cite{TV} on the Bernstein-transform datum
$F_0$, recorded here so that the present note is self-contained.

\begin{assumption}[Conditions on the initial datum,
		after~\cite{TV}]\label{ass:data}
	There exist $\beta\in(0,1)$ and $C>0$ such that
	\begin{enumerate}[label=\textup{(A\arabic*)},ref=\textup{(A\arabic*)},leftmargin=3.2em]
		\item\label{A1} $0\le F_0'(x)\le m$ and $F_0'(0)=m$,
		\item\label{A2} $-C\le F_0''(x)\le0$,
		\item\label{A3} $\displaystyle-\frac{m}{e}\le xF_0''(x)\le0$
		      \quad and\quad
		      $\bigl\|xF_0''\bigr\|_{C^{0,\beta}([0,\infty))}\le C$.
	\end{enumerate}
\end{assumption}

\noindent These conditions follow whenever $F_0$ is the Bernstein transform of a
nonnegative measure $c_0$ with mass $m_1(0)=m$ and bounded second moment
$m_2(0)=\int_0^\infty s^2\,c_0(s)\dd s<\infty$. Indeed, differentiating
$F_0(x)=\int_0^\infty(1-e^{-sx})c_0(s)\dd s$ gives
$xF_0''(x)=-\int_0^\infty s\,\varphi(sx)\,c_0(s)\dd s$ with
$\varphi(z)=ze^{-z}\in[0,e^{-1}]$, whence $-m/e\le xF_0''\le0$. Moreover
$\varphi$ is $1$-Lipschitz and bounded, so
$|\varphi(sx)-\varphi(sy)|\le(2/e)^{1-\beta}\bigl(s\,|x-y|\bigr)^{\beta}$
for any
$\beta\in(0,1)$. Hence
\[
	\bigl|xF_0''(x)-yF_0''(y)\bigr|
	\le(2/e)^{1-\beta}\,|x-y|^{\beta}\int_0^\infty s^{1+\beta}c_0(s)\dd s,
\]
and $\int_0^\infty s^{1+\beta}c_0\,\dd s\le m^{1-\beta}\,m_2(0)^{\beta}<\infty$
by interpolation, which bounds $\|xF_0''\|_{C^{0,\beta}}$ and
gives~\ref{A3}. In Theorem~\ref{thm:main} we additionally assume
$m_0(0)<\infty$,
which gives $F_0\in L^\infty$ and is used in the localization at
spatial infinity
(Proposition~\ref{prop:visc-prop}).

\subsection{Main idea}
Introduce
\begin{equation}\label{eq:Mq}
	M(x,t)=mx-F(x,t),\qquad q(x,t)=M_x(x,t)=m-F_x(x,t).
\end{equation}
A direct substitution (Section~\ref{sec:transformed}) turns \eqref{eq:HJ} into
the clean form
\begin{equation}\label{eq:Meq}
	\,M_t=\tfrac12\,q(q+1)-\frac{M}{x}\,,
	\qquad q=M_x .
\end{equation}
The estimate we propagate is the convexity-type, one-sided bound
\begin{equation}\label{eq:invariant}
	\,W:=2M-xM_x \ge 0 \,
	\qquad\Longleftrightarrow\qquad
	\frac{M}{x}\ge\frac12\,M_x\,.
\end{equation}
A function for which the property $W\ge0$ holds for all time is said to
possess the \emph{half-slope invariant}.
This is exactly the inequality that the curvature-barrier
argument of~\cite{TV} did not exploit. We show three things. First,
\eqref{eq:invariant} holds at $t=0$ for any
Bernstein-transform datum (Lemma~\ref{lem:initial}). Second, it is
propagated by the
inviscid flow (Lemma~\ref{lem:inviscid}) and, rigorously, by the viscous
$\delta$-regularization of~\cite{TV} (Proposition~\ref{prop:visc-prop} and
Corollary~\ref{cor:eps}). Third, it upgrades the bound on the dangerous
coefficient in the barrier from $B\le2M_x$ to $B\le M_x$, which is what the
$\tfrac12$-threshold was paying for (Section~\ref{sec:removal}).

\subsection{A curious Keller--Segel analogue}
The half-slope invariant has a curious analogue in the radial Keller--Segel
equation, where the same one-sided quantity satisfies a linear parabolic
equation with no zeroth-order term, see \eqref{eq:KSH}. This is
purely an aside, recorded in
Section~\ref{sec:legendre}. It plays no role in the proof.

\subsection{Organization}
Section~\ref{sec:transformed} carries out the substitution $M=mx-F$ and records
the resulting equation for $M$. Section~\ref{sec:initial} discusses the
half-slope invariant of $M$: it is verified at $t=0$ for
Bernstein-transform data and
then shown to propagate, first under the inviscid flow, and then, as a genuine
maximum-principle estimate, under the viscous $\delta$-regularization
of~\cite{TV}. Section~\ref{sec:removal} uses the invariant to sharpen the bound
on the dangerous coefficient to $B\le M_x$, thereby extending the curvature
barrier $xF_{xx}>-1$ to the entire range $0<m\le1$.
Section~\ref{sec:completion} proves Theorem~\ref{thm:main}. Finally,
Section~\ref{sec:legendre} records the Keller--Segel analogue of the half-slope
invariant.

\section{The transformed equation}\label{sec:transformed}

Start from \eqref{eq:HJ} and set $M=mx-F$, so that
\[
	F=mx-M,\qquad F_x=m-M_x,\qquad F_t=-M_t .
\]
Then $F_x-m=-M_x$ and $F_x-m-1=-M_x-1$, while
$\tfrac{F}{x}-m=\tfrac{mx-M}{x}-m=-\tfrac{M}{x}$. Substituting,
\[
	-M_t+\tfrac12(-M_x)(-M_x-1)-\frac{M}{x}=0,
\]
which is \eqref{eq:Meq}. With $q=M_x$ this reads $M_t=\tfrac12 q(q+1)-M/x$.

For Bernstein data the new unknown has a transparent representation. If
$F_0(x)=\int_0^\infty(1-e^{-sx})c_0(s)\dd s$ and $m=\int_0^\infty
	s\,c_0(s)\dd s$,
then
\begin{equation}\label{eq:M0}
	M_0(x)=mx-F_0(x)=\int_0^\infty\bigl(sx-1+e^{-sx}\bigr)c_0(s)\dd s,
\end{equation}
\begin{equation}\label{eq:M0x}
	M_{0,x}(x)=m-F_{0,x}(x)=\int_0^\infty s\bigl(1-e^{-sx}\bigr)c_0(s)\dd s .
\end{equation}
Because $sx\ge1-e^{-sx} \geq 0$ and $1-e^{-sx}\geq 0$,
it follows that $M_0\ge0$ and $M_{0,x}\ge0$.
The latter is the statement $0\le F_{0,x}\le m$.
These representations are the basis for the initial invariant.

\section{The half-slope invariant}\label{sec:initial}
\subsection{Invariant propagation for inviscid flow}
We first note the following fact at $t=0$.
\begin{lemma}
	\label{lem:initial}
	Let $F_0$ be the Bernstein transform of a nonnegative measure $c_0$ on
	$(0,\infty)$ with finite first moment $m=\int_0^\infty s\,c_0(s)\dd s<\infty$.
	Then $M_0=mx-F_0$ satisfies
	\[
		2M_0(x)-x\,M_{0,x}(x)\ge0\qquad\text{for all }x\ge0 .
	\]
\end{lemma}

\begin{proof}
	From \eqref{eq:M0}--\eqref{eq:M0x},
	\[
		2M_0-x M_{0,x}
		=\int_0^\infty\Bigl[\,2(sx-1+e^{-sx})-x\,s(1-e^{-sx})\,\Bigr]c_0(s)\dd s
		=\int_0^\infty\Bigl[\,sx-2+(sx+2)e^{-sx}\,\Bigr]c_0(s)\dd s .
	\]
	Set $z=sx\ge0$ and $\psi(z)=z-2+(z+2)e^{-z}$. Then $\psi(0)=0$ and
	\[
		\psi'(z)=1-(z+1)e^{-z}\ge0\qquad(z\ge0),
	\]
	because $z+1\le e^{z}$ for all $z\ge0$. Hence $\psi\ge0$ on
	$[0,\infty)$, and the
	integrand $\psi(sx)c_0(s)$ is nonnegative. Therefore $2M_0-xM_{0,x}\ge0$.
\end{proof}

Throughout, we denote
\begin{equation}
	\label{eq:W}
	W=2M-xM_x \qquad  \text{and} \qquad q=M_x \,.
\end{equation}
We then have the following identities
\begin{equation}\label{eq:Wderivs}
	W_x = M_x-xM_{xx} = q-xq_x \,,
	\qquad
	W_{xx} = -x\,M_{xxx} = -x\,q_{xx} \,.
\end{equation}

\begin{lemma}\label{lem:inviscid}
	Let $M$ be a smooth solution of \eqref{eq:Meq} on a region of $(0,\infty)^2$.
	Then
	\begin{equation}\label{eq:invisc-id}
		\,W_t-\Bigl(M_x+\tfrac12\Bigr)W_x+\frac{3}{2x}\,W=0\,.
	\end{equation}
	Consequently, as long as the solution stays smooth and its characteristics
	remain in $(0,\infty)$,
	$W(\cdot,0)\ge0$ implies $W(\cdot,t)\ge0$\,.
\end{lemma}

\begin{proof}
	Differentiating \eqref{eq:Meq} in $x$ gives
	$q_t=(q+\tfrac12)q_x-q/x+M/x^2$.
	Therefore,
	\[
		W_t=2M_t-xq_t
		=q^2+2q-\frac{3M}{x}-x\Bigl(q+\tfrac12\Bigr)q_x .
	\]
	Since $\bigl(q+\tfrac12\bigr)W_x=\bigl(q+\tfrac12\bigr)q
		-x\bigl(q+\tfrac12\bigr)q_x$, we get
	\begin{equation*}
		W_t-(q+\tfrac12)W_x=\tfrac32 q-\tfrac{3M}{x}=-\tfrac{3}{2x}W \,,
	\end{equation*}
	which is~\eqref{eq:invisc-id}. Along a characteristic $\dot
		X=-(q+\tfrac12)$ one has
	\begin{equation*}
		\dot W=-\tfrac{3}{2X}W\,.
	\end{equation*}
	As a consequence, $W$ retains its sign along every characteristic that
	remains in $(0,\infty)$, where the coefficient $3/(2X)$ stays finite.
\end{proof}

\begin{remark}
	When $q\ge0$ the characteristics move to the left
	with speed at least $\tfrac12$, so each of them reaches $x=0$ in finite
	time, and $3/(2X)$ blows up there. Sign preservation survives, because the
	backward-in-time characteristic through any point of $(0,\infty)^2$ moves
	away from $x=0$ and stays in $(0,\infty)$ up to $t=0$.
\end{remark}
\subsection{Viscous approximations}
In~\cite{TV}, the authors construct the regular solution as the
vanishing-viscosity limit of
\begin{equation}\label{eq:visc-F}
	\begin{gathered}
		F_t + \tfrac12 (F_x-m)(F_x-m-1) + \frac{F}{x}-m =
		A_{\eps,\delta}(x)\,F_{xx},
		\qquad A_{\eps,\delta}(x)=\eps\,a(x)+\delta,\\
		F(x,0)=F_0(x),\qquad F(0,t)=0,
	\end{gathered}
\end{equation}
where $\eps>0$, $\delta\ge0$, and $a\in C^\infty([0,\infty))$ is nonnegative,
nondecreasing and concave with
\begin{equation}\label{eq:cutoff}
	a(x)=x \ \text{ on }[0,1],\qquad a(x)=2 \ \text{ on }[3,\infty) \,.
\end{equation}
For clarity, we will always write $A(x)$ for $A_{\eps,\delta}(x)$.
In
the $M$-variable, \eqref{eq:visc-F} becomes
\begin{equation}\label{eq:visc-M}
	M_t=A(x)\,M_{xx}+\tfrac12\,M_x(M_x+1)-\frac{M}{x}.
\end{equation}

\begin{lemma}\label{lem:viscous}
	Let $M$ solve \eqref{eq:visc-M} classically. Then,
	with the notations~\eqref{eq:W},
	\begin{equation}\label{eq:visc-id}
		W_t-\Bigl(q+\tfrac12\Bigr)W_x-A\,W_{xx}
		+\frac{S}{x}\,W_x+\frac{3}{2x}\,W
		=\frac{S}{x}\,q\,,
	\end{equation}
	where $S(x):=2A(x)-xA'(x)$.
\end{lemma}

\begin{proof}
	Differentiating \eqref{eq:visc-M} in $x$,
	$q_t=A'q_x+Aq_{xx}+(q+\tfrac12)q_x-q/x+M/x^2$. Using
	$W_t=2M_t-xq_t$, $W_x=q-xq_x$ and $W_{xx}=-xq_{xx}$, a direct
	computation gives
	\[
		W_t=q^2+2q-\frac{3M}{x}+(2A-xA')q_x-xAq_{xx}-x\Bigl(q+\tfrac12\Bigr)q_x ,
	\]
	and assembling the stated combination, all $q_x,q_{xx},q^2,M/x$ terms cancel,
	leaving \eqref{eq:visc-id}.
\end{proof}

\begin{lemma}
	\label{lem:cutoff}
	Let $a$ be as in \eqref{eq:cutoff}.
	One has $2a(x)-xa'(x)\ge0$ for all $x\ge0$.
	Consequently, for $A=\eps a+\delta$,
	\[
		S(x)=2A-xA'=\eps\,(2a-xa')+2\delta\ \ge\ 0 .
	\]
\end{lemma}

\begin{proof}
	Let $g(x)=2a(x)-xa'(x)$. Then $g(0)=2a(0)=0$ (as $a(x)=x$ near $0$) and
	$g'(x)=a'(x)-x\,a''(x)\ge0$, because $a'\ge0$ ($a$ nondecreasing), $a''\le0$
	($a$ concave) and $x\ge0$. Thus $g$ is nondecreasing with $g(0)=0$,
	so $g\ge0$.
	Adding the $\delta$-layer only increases $S$ by $2\delta\ge0$.
\end{proof}

\subsection{Invariant propagation for the viscous approximations}

We now prove $W\ge0$ as a genuine maximum-principle estimate, first for the
uniformly parabolic problem ($\delta>0$).

\begin{proposition}[Propagation for $F^{\eps,\delta}$]\label{prop:visc-prop}
	Let $0<m\le1$, $\eps>0$, $\delta>0$, $A=\eps a+\delta$ with $a$ as in
	\eqref{eq:cutoff}, and let $F_0$ be bounded and satisfy \ref{A1}--\ref{A2}
	with $F_0(0)=0$. Let $F^{\eps,\delta}$ be the
	classical solution of
	\eqref{eq:visc-F} with $F^{\eps,\delta}(0,t)=0$,
	$F^{\eps,\delta}(\cdot,0)=F_0$.
	Put $M=mx-F^{\eps,\delta}$, $q=M_x$, $W=2M-xM_x$. If $W(\cdot,0)\ge0$, then
	\[
		\,W(x,t)\ge0\quad\text{for all }x>0,\ t\ge0 \,.
	\]
\end{proposition}

The proof of Proposition~\ref{prop:visc-prop} rests on a gradient bound for
$F^{\eps,\delta}$. In~\cite{TV} this bound is their Lemma~3.5. The statement
there concerns the degenerate approximation obtained after $\delta\downarrow0$,
and the uniformly parabolic case $\delta>0$ that we use is treated inside its
proof. For self-consistency we give a complete proof, following the scheme of
Lemma~3.1 of~\cite{TV} with the localization made explicit.

\begin{lemma}[Gradient bound for the uniformly parabolic
		problem]\label{lem:Fxbound}
	Let $0<m\le1$, $\eps>0$, $\delta>0$, and let $F_0$ satisfy \ref{A1}
	and~\ref{A2} with $F_0(0)=0$. Let $F^{\eps,\delta}$ be the classical
	solution of \eqref{eq:visc-F}.
	Then $0\le F^{\eps,\delta}\le mx$ and
	\begin{equation}\label{eq:Fxbound}
		0\le F^{\eps,\delta}_x\le m \,.
	\end{equation}
\end{lemma}

\begin{proof}
	Write $F=F^{\eps,\delta}$ and $p=F^{\eps,\delta}_x$, and fix a strip
	$[0,\infty)\times[0,T]$. By \ref{A1} and $F_0(0)=0$, integration gives
	$0\le F_0(x)\le mx$. From the theory of parabolic equations~\cite{LSU}, $F$
	is smooth on $(0,\infty)^2$, and both $F$ and $p$ are continuous on the
	strip with $\Lambda_T:=\sup_{[0,\infty)\times[0,T]}|p|<\infty$.

	\emph{Step 1. We show $0\le F\le mx$.}
	Write
	\[
		N[\phi]:=\phi_t+\tfrac12(\phi_x-m)(\phi_x-m-1)+\frac{\phi}{x}-m
		-A\,\phi_{xx},
	\]
	so that $N[F]=0$ and $N[mx]=0$, while $m\le1$ gives
	$N[0]=\tfrac12 m(m-1)\le0$. Let $w$ denote either $0-F$ or $F-mx$. In both
	cases $w=u-v$ with $N[u]\le N[v]$, and the mean value theorem applied to the
	quadratic term yields
	\[
		w_t+b\,w_x+\frac{w}{x}-A\,w_{xx}\le0,
		\qquad b(x,t)=\xi(x,t)-m-\tfrac12,
	\]
	where $\xi(x,t)$ lies between $u_x$ and $v_x$. Since $u_x$ and $v_x$ take
	the values $0$, $p$ or $m$, we get $|b|\le\Lambda_T+2m+1=:\bar b_0$.
	Moreover, $w\le0$ on $\{t=0\}$, $w=0$ on $\{x=0\}$, and
	$|w|\le(\Lambda_T+m)\,x$.

	We claim that $w \leq 0$. Set
	$h(x,t):=e^{\lambda t}(1+x^{2})$ with $\lambda:=\bar b_0+2(2\eps+\delta)$.
	Using $2x\le1+x^{2}$, $A\le2\eps+\delta$ and $e^{\lambda t}\le h$,
	\[
		h_t+b\,h_x+\frac{h}{x}-A\,h_{xx}
		\ \ge\ \lambda h-\bar b_0(1+x^{2})e^{\lambda t}
		-2(2\eps+\delta)e^{\lambda t}\ \ge\ 0 .
	\]
	Fix $\gamma>0$ and set $U:=w-\gamma h$, so that
	$U_t+b\,U_x+U/x-A\,U_{xx}\le0$. As $|w|\le(\Lambda_T+m)x$, $U\to-\infty$
	when $x\to\infty$ uniformly in $t$, so $U$ attains its maximum over the
	strip at some $(x_0,t_0)$. If this maximum were positive, then $x_0>0$ and
	$t_0>0$, because $U\le w\le0$ on $\{t=0\}$ and $U=-\gamma e^{\lambda t}<0$
	on $\{x=0\}$. At such a maximum $U_t\ge0$, $U_x=0$, $U_{xx}\le0$ and
	$U/x_0>0$, whence $U_t+b\,U_x+U/x_0-A\,U_{xx}>0$, a contradiction. Hence
	$U\le0$, and $\gamma\downarrow0$ gives $w\le0$, which proves
	$0\le F\le mx$.

	\emph{Step 2. We identify boundary values of $p$.}
	By Step 1, $0\le F(x,t)/x\le m$, so $F(0,t)=0$ and continuity of $p$ up to
	$x=0$ give $p(0,t)=\lim_{x\to0^{+}}F(x,t)/x\in[0,m]$. Together with
	\ref{A1} at $t=0$, $0\le p\le m$ on the parabolic boundary
	$\{t=0\}\cup\{x=0\}$ of the strip.

	\emph{Step 3.}
	Introduce the linear parabolic operator
	\[
		\Lop^{\eps,\delta}\phi:=\phi_t+\Bigl(p-m-\tfrac12-A'\Bigr)\phi_x
		-A\,\phi_{xx},
	\]
	Differentiating \eqref{eq:visc-F} in $x$ and writing, by the mean value
	theorem and $F(0,t)=0$, $F(x,t)=x\,p(\theta x,t)$ with
	$\theta=\theta(x,t)\in(0,1)$, we get
	\[
		\Lop^{\eps,\delta}p+\frac{p(x,t)-p(\theta x,t)}{x}=0
		\qquad\text{for }x>0,\ 0<t\le T .
	\]
	Since $0\le a'\le1$ by \eqref{eq:cutoff} and concavity, the drift
	coefficient is bounded, with $|p-m-\tfrac12-A'|\le\Lambda_T+m+1+\eps
		=:\bar b$.

	\emph{Step 4. We show $p\le m$.}
	Fix $\sigma>0$ and $0<\nu<\sigma/\bar b$, and consider
	$\varphi:=p-\nu x-\sigma t$ on the strip. As $\varphi\to-\infty$ when
	$x\to\infty$, the maximum of $\varphi$ is attained at some $(x_0,t_0)$. If
	$x_0>0$ and $t_0>0$, then $p_t\ge\sigma$, $p_x=\nu$, $p_{xx}\le0$, and
	$\varphi(x_0,t_0)\ge\varphi(\theta x_0,t_0)$ gives
	$p(x_0,t_0)-p(\theta x_0,t_0)\ge0$. The first three facts and the drift
	bound give $\Lop^{\eps,\delta}p\,(x_0,t_0)\ge\sigma-\bar b\,\nu$, the fourth
	makes the nonlocal term nonnegative, and the equation of Step 3 then forces
	\[
		0\ =\ \Lop^{\eps,\delta}p\,(x_0,t_0)
		+\frac{p(x_0,t_0)-p(\theta x_0,t_0)}{x_0}
		\ \ge\ \sigma-\bar b\,\nu\ >\ 0,
	\]
	a contradiction. Hence the maximum is attained
	on $\{t=0\}\cup\{x=0\}$, where $\varphi\le p\le m$ by Step 2. Therefore
	$p\le m+\nu x+\sigma t$, and letting $\nu\downarrow0$ and then
	$\sigma\downarrow0$ gives $p\le m$.

	\emph{Step 5. We show $p\ge0$.}
	Symmetrically, the minimum of $\psi:=p+\nu x+\sigma t$ over the strip is
	attained. At a minimum point with $x_0>0$ and $t_0>0$, the reversed
	inequalities $p_t\le-\sigma$, $p_x=-\nu$, $p_{xx}\ge0$ and
	$p(x_0,t_0)-p(\theta x_0,t_0)\le0$ turn the equation of Step 3 into
	$0\le-\sigma+\bar b\,\nu<0$, a contradiction. Hence the minimum is attained
	on $\{t=0\}\cup\{x=0\}$, where $\psi\ge p\ge0$ by Step 2, and letting
	$\nu\downarrow0$ and then $\sigma\downarrow0$ gives $p\ge0$.
\end{proof}

\begin{proof}[Proof of Proposition~\ref{prop:visc-prop}]
	Write the linear operator on the left of \eqref{eq:visc-id} as
	\[
		\Lop\phi:=\phi_t-\Bigl(q+\tfrac12\Bigr)\phi_x+\frac{S}{x}\phi_x
		-A\,\phi_{xx}+\frac{3}{2x}\phi,
		\qquad\text{so that}\qquad \Lop W=\frac{S}{x}\,q .
	\]
	By Lemma~\ref{lem:cutoff}, $S\ge0$, and by \eqref{eq:Fxbound},
	$q=m-F^{\eps,\delta}_x\ge0$.
	Therefore, $\Lop W=\tfrac{S}{x}q\ge0$. The zeroth-order coefficient
	$3/(2x)$ of $\Lop$
	is nonnegative and the second-order coefficient $-A\le0$, so $\Lop$
	is a proper
	uniformly parabolic operator on $[\eta,R]$  for the minimum principle.

	Fix $T>0$.

	\emph{Step 1.}
	We first claim that
	\[
		0\ \le\ F^{\eps,\delta}(x,t)\ \le\ K_T:=\|F_0\|_{L^\infty}+mT
		\qquad\text{on }[0,\infty)\times[0,T] .
			\]
			The lower bound, together with $F^{\eps,\delta}\le mx$, is part of
			Lemma~\ref{lem:Fxbound}.
			For the upper bound, \eqref{eq:Fxbound} gives $F^{\eps,\delta}_x-m\le0$ and
			$F^{\eps,\delta}_x-m-1\le-1<0$, so
			$(F^{\eps,\delta}_x-m)(F^{\eps,\delta}_x-m-1)\ge0$. Since also
			$F^{\eps,\delta}/x\ge0$ by the lower bound, equation~\eqref{eq:visc-F} yields
			the differential inequality
			\[
			F^{\eps,\delta}_t\ \le\ A\,F^{\eps,\delta}_{xx}+m
			\qquad\text{in }(0,\infty)\times(0,T] .
	\]
	Hence $V:=F^{\eps,\delta}-\|F_0\|_{L^\infty}-mt$ satisfies
	\[
		V_t\le A\,V_{xx},\qquad
		V(\cdot,0)\le0,\qquad V(0,\cdot)\le0,\qquad V\le mx .
	\]
	Here the linear growth $V\le mx$ comes from $F^{\eps,\delta}\le mx$, and $A$
	is bounded with $0<\delta\le A\le2\eps+\delta$. We claim that these facts
	force $V\le0$ on $[0,\infty)\times[0,T]$. Fix $\gamma>0$ and set
	$U:=V-\gamma\bigl(x^{2}+2(2\eps+\delta)t\bigr)$. Then $U_t\le A\,U_{xx}$,
	because $A\le2\eps+\delta$. Moreover, $U\le0$ on $\{t=0\}$ and on $\{x=0\}$,
	and $U\le mx-\gamma x^{2}\le0$ wherever $x\ge m/\gamma$. The classical
	maximum principle on the rectangle $[0,m/\gamma]\times[0,T]$ then forces
	$U\le0$ on all of $[0,\infty)\times[0,T]$. Letting $\gamma\downarrow0$ gives
	$V\le0$. Thus $0\le F^{\eps,\delta}\le K_T$, as claimed.

	Consequently, at the right edge, using
	$W=2(mx-F)-x(m-F_x)=mx-2F+xF_x$ and $F_x\ge0$,
	\[
		W(x,t)\ge mx-2K_T,
		\qquad\text{so}\qquad W(R,t)\ge0\ \ \text{whenever }R\ge\frac{2K_T}{m}.
	\]

	\emph{Step 2.} Since $M(0,t)=0$ and $M_x=q\in[0,m]$,
	\[
		M(x,t)=\int_0^x q(y,t)\dd y\in[0,mx] \,.
	\]
	Therefore,
	\begin{equation*}
		W(x,t)=2M-xq\ge-xq\ge-mx .
	\end{equation*}
	In particular $W(\eta,t)\ge-m\eta$ for every $\eta>0$.

	\emph{Step 3 (localized minimum principle).} Fix $0<\eta<R$ with
	$R\ge 2K_T/m$,
	and set $Y:=W+m\eta$. By Steps~1--2, $Y\ge0$ on the parabolic
	boundary $\partial Q_{\eta,R}:=\partial((\eta,R)\times(0,T])$.
	Indeed $Y(x,0)=W(x,0)+m\eta\ge0$,
	$Y(\eta,t)\ge-m\eta+m\eta=0$, and $Y(R,t)\ge0$. Moreover,
	\[
		\Lop Y=\Lop W+\frac{3}{2x}\,m\eta=\frac{S}{x}q+\frac{3}{2x}m\eta\ge0 .
	\]
	If $Y$ had a negative interior minimum at $(x_0,t_0)\in Q_{\eta,R}$, then
	$Y(x_0,t_0)<0$, $Y_t(x_0,t_0)\le0$, $Y_x(x_0,t_0)=0$,
	$Y_{xx}(x_0,t_0)\ge0$, so
	\[
		\Lop Y(x_0,t_0)=Y_t-A\,Y_{xx}+\frac{3}{2x_0}Y
		\le\frac{3}{2x_0}\,Y(x_0,t_0)<0,
	\]
	contradicting $\Lop Y\ge0$. Hence $Y\ge0$ on $\overline{Q_{\eta,R}}$, i.e.
	$W\ge-m\eta$ on $[\eta,R]\times[0,T]$. Letting $\eta\to0^+$ gives $W\ge0$ on
	$(0,R)\times[0,T]$.
	Since also $W\ge0$ for $x\ge R$ (Step~1) and
	$R$ may be taken
	arbitrarily large, $W\ge0$ on $(0,\infty)\times[0,T]$. As $T>0$ was arbitrary,
	the proof is complete.
\end{proof}

\begin{corollary}[Propagation for $F^{\eps}$]\label{cor:eps}
	Let $F^{\eps}=\lim_{\delta\to0^+}F^{\eps,\delta}$, and
	$M^\eps=mx-F^{\eps}$. Then
	\[
		2M^\eps-xM^\eps_x\ge0\quad\text{on }(0,\infty)\times[0,\infty) \,.
	\]
\end{corollary}

\begin{proof}
	Proposition~\ref{prop:visc-prop} gives
	$2M^{\eps,\delta}-xM^{\eps,\delta}_x\ge0$
	for every $\delta>0$.
	On compact subsets of $(0,\infty)\times(0,\infty)$ the
	equation is uniformly parabolic as $\delta\to0$ (because $a(x)>0$ there), so
	$F^{\eps,\delta}\to F^{\eps}$ locally uniformly.

	Furthermore, each $F^{\eps,\delta}$ is concave in $x$
	(\cite{TV}, proof of Lemma~3.5, where the concavity is established at the
	level of the $\delta$-problem)
	and the limit $F^{\eps}$ is $C^1$ in $x$ (\cite{TV}, Lemma~3.6).
	By the convex analysis fact that locally uniform convergence of
	concave functions
	implies convergence of the derivatives at every point of
	differentiability of the limit,
	$F^{\eps, \delta}_x \to F^{\eps}_x$ pointwise on $(0,\infty)$.

	Passing to the limit in
	$2(mx-F^{\eps,\delta})-x(m-F^{\eps,\delta}_x)$ yields the claim.
\end{proof}

\section{Removal of the $m<\tfrac12$ obstruction}
\label{sec:removal}

\subsection{The improved barrier}
Throughout this section $F^\eps$ denotes the degenerate viscous approximation
of~\cite{TV} obtained from~\eqref{eq:visc-F} after $\delta\downarrow0$, namely
the classical solution of
\begin{equation}\label{eq:visc-eps}
	F^\eps_t+\tfrac12(F^\eps_x-m)(F^\eps_x-m-1)+\frac{F^\eps}{x}-m
	=\eps\,a(x)\,F^\eps_{xx},
	\qquad F^\eps(x,0)=F_0,\quad F^\eps(0,t)=0,
\end{equation}
with $a$ the cutoff~\eqref{eq:cutoff}. To keep the regularization visible we
write
\[
	G^\eps=xF^\eps_{xx},\qquad M=mx-F^\eps,\qquad q=M_x=m-F^\eps_x,\qquad
	W=2M-xM_x .
\]
From Lemma 3.5 and Lemma 3.6 of~\cite{TV},
\begin{equation}\label{eq:sec4-recall}
	0\le F^\eps_x\le m,\qquad F^\eps\ \text{is concave in }x,\qquad
	G^\eps\le0,\qquad G^\eps(0,t)=0,
\end{equation}
while Corollary~\ref{cor:eps} supplies the propagated invariant
\begin{equation}\label{eq:sec4-inv}
	W=2M-xM_x\ge0,\qquad\text{equivalently}\qquad \frac{M}{x}\ge\frac{M_x}{2}.
\end{equation}
Finally, hypothesis~\ref{A3} controls the initial curvature: for Bernstein data
$re^{-r}\le e^{-1}$ gives $xF_0''(x)\ge-m/e$, so for $0<m\le1$
\begin{equation}\label{eq:initcurv}
	G^\eps(\cdot,0)=xF_0''\ge-\frac{m}{e}>-1 .
\end{equation}

The following theorem replaces Lemma 3.7 of~\cite{TV}, extending their curvature
barrier from $0<m<\tfrac12$ to the full critical range $0<m\le1$.
The gain comes from the invariant~\eqref{eq:sec4-inv}, which sharpens
the zeroth-order
coefficient bound from $B\le2M_x$ to $B\le M_x$.

\begin{theorem}[Improved curvature barrier]\label{thm:barrier}
	Let $0<m\le1$, and let $F^\eps$ be the degenerate viscous
	approximation~\eqref{eq:visc-eps} associated with Bernstein-transform data
	satisfying the hypotheses of Theorem~\ref{thm:main}. If $\eps<\tfrac14$, then
	\[
		-1< xF^\eps_{xx}\le0\qquad\text{on }(0,\infty)^2 .
	\]
\end{theorem}

\begin{proof}
	We run a localized minimum principle for $G^\eps=xF^\eps_{xx}$,
	following~\cite{TV}, Lemma 3.7 and Remark 3.8, but with the
	invariant~\eqref{eq:sec4-inv} in place of the crude bound used there.

	\emph{Step 1.}
	Put $\alpha(T)=\inf_{(x,t)\in[0,\infty)\times[0,T]}G^\eps(x,t)$. As
	in~\cite{TV}, Lemma 3.7, $G^\eps=xF^\eps_{xx}$ is bounded and Hölder
	continuous on $[0,\infty)\times[0,T]$ and satisfies $G^\eps\le0$,
	$G^\eps(0,t)=0$ (cf.~\eqref{eq:sec4-recall}).
	Hence $\alpha$ is continuous and
	nonincreasing in $T$, and by~\eqref{eq:initcurv}, $\alpha(0)\ge-m/e>-1$. It
	therefore suffices to rule out a first time $T$ with $\alpha(T)=-1$.

	\emph{  We proceed by contradiction and suppose that there is a first time
		$T$ where $\alpha(T) = -1$.}

	\emph{Step 2.}
	Differentiating~\eqref{eq:visc-eps} twice in $x$,
	\begin{equation}\label{eq:312}
		\d_tF^\eps_{xx}+\Bigl(F^\eps_x-\bigl(m+\tfrac12\bigr)\Bigr)F^\eps_{xxx}
		+(F^\eps_{xx})^2
		+\frac{F^\eps_{xx}}{x}-\frac{2F^\eps_x}{x^2}+\frac{2F^\eps}{x^3}
		=\eps\bigl(a''F^\eps_{xx}+2a'F^\eps_{xxx}+aF^\eps_{xxxx}\bigr).
	\end{equation}

	\emph{Step 3.}
	Fix a large $k$. Choose $(y_k,s_k)\in[0,\infty)\times[0,T]$ with
	\[
		G^\eps(y_k,s_k)\le\alpha(T)+\tfrac{1}{2k},
	\]
	and then $0<\rho_k<\tfrac1k$ so small that $\rho_k y_k\le\tfrac{1}{2k}$. Here
	$\rho_k$ is the localization parameter, distinct from the already-removed
	regularization $\delta$. Since $G^\eps$ is bounded below and $\rho_k
		x\to+\infty$ as
	$x\to\infty$, the function $G^\eps(x,t)+\rho_k x$ attains its minimum over
	$[0,\infty)\times[0,T]$ at some point $(x_k,t_k)$. By minimality
	and the choice
	of $(y_k,s_k)$,
	\[
		\alpha(T)\ \le\ G^\eps(x_k,t_k)+\rho_k x_k\ \le\ G^\eps(y_k,s_k)+\rho_k y_k
		\ \le\ \alpha(T)+\tfrac1k,
	\]
	and since $G^\eps(x_k,t_k)\ge\alpha(T)$ this also gives $\rho_k
		x_k\le\tfrac1k$.

	Furthermore, $G^\eps+\rho_k x$ equals $0$ at $x=0$ and tends to
	$+\infty$ as
	$x\to\infty$, while on $\{t=0\}$ its values are
	at least $-m/e>\alpha(T)+\tfrac1k$ for
	$k$ large. So, $x_k,t_k>0$.

	Set $\alpha_k=G^\eps(x_k,t_k)$, so
	$\alpha_k\to-1$. At
	the localized minimum
	\[
		G^\eps_t\le0,\qquad G^\eps_x=-\rho_k,\qquad G^\eps_{xx}\ge0 .
	\]
	We translate these into pointwise relations for the derivatives of $F^\eps$.
	Since $G^\eps=xF^\eps_{xx}$,
	\[
		G^\eps_x=F^\eps_{xx}+x\,F^\eps_{xxx},
		\qquad
		G^\eps_{xx}=2F^\eps_{xxx}+x\,F^\eps_{xxxx},
	\]
	and, by the definition $\alpha_k=G^\eps(x_k,t_k)=x_kF^\eps_{xx}$, we have
	$F^\eps_{xx}=\alpha_k/x_k$ at $(x_k,t_k)$. Evaluating the first-order
	condition $G^\eps_x=-\rho_k$ there gives
	$F^\eps_{xx}+x_kF^\eps_{xxx}=-\rho_k$.
	Multiplying by $x_k$ and using $x_kF^\eps_{xx}=\alpha_k$,
	\begin{equation}\label{eq:locmin}
		x_k^2F^\eps_{xxx}=-x_kF^\eps_{xx}-\rho_k x_k=-\alpha_k-\rho_k x_k .
	\end{equation}
	Similarly the second-order condition $G^\eps_{xx}\ge0$ reads
	$2F^\eps_{xxx}+x_kF^\eps_{xxxx}\ge0$,
	i.e.\ $x_kF^\eps_{xxxx}\ge-2F^\eps_{xxx}$.
	Multiplying by $x_k$ and inserting~\eqref{eq:locmin},
	\begin{equation}\label{eq:locmin2}
		x_k^2F^\eps_{xxxx}\ \ge\ -2x_kF^\eps_{xxx}
		=\frac{2\bigl(\alpha_k+\rho_k x_k\bigr)}{x_k}
		=2F^\eps_{xx}+2\rho_k .
	\end{equation}

	\emph{Step 4.}
	Multiply~\eqref{eq:312} by $x_k^2$ and evaluate at $(x_k,t_k)$. On the
	left-hand side, $x_k^{2}\,\d_tF^\eps_{xx}=x_kG^\eps_t\le0$ may be dropped,
	\eqref{eq:locmin} turns the third-derivative term into
	$(m+\tfrac12-F^\eps_x)(\alpha_k+\rho_kx_k)$, and the remaining terms are
	$x_k^{2}(F^\eps_{xx})^{2}=\alpha_k^{2}$, $x_kF^\eps_{xx}=\alpha_k$, and
	$-2F^\eps_x+2F^\eps/x_k=B_k$. On the right-hand side, \eqref{eq:locmin},
	\eqref{eq:locmin2} and $a\ge0$ give
	\[
		\eps\bigl(a''x_k^{2}F^\eps_{xx}+2a'x_k^{2}F^\eps_{xxx}
		+a\,x_k^{2}F^\eps_{xxxx}\bigr)
		\ \ge\
		\eps\,\alpha_k\Bigl(\frac{2a}{x_k}-2a'+x_ka''\Bigr)
		-2\eps a'\rho_kx_k+2\eps a\rho_k ,
	\]
	with $a$, $a'$, $a''$ evaluated at $x_k$. Discarding the nonnegative term
	$2\eps a(x_k)\rho_k$ and moving $-2\eps a'\rho_kx_k$ to the left-hand side,
	we obtain
	\begin{equation}\label{eq:locineq}
		\alpha_k^2+\alpha_k\Bigl(m+\tfrac32-F^\eps_x\Bigr)+B_k
		+\rho_k x_k\Bigl(m+\tfrac12+2\eps a'(x_k)-F^\eps_x\Bigr)
		\ \ge\
		\eps\,\alpha_k\Bigl(\frac{2a(x_k)}{x_k}-2a'(x_k)+x_k a''(x_k)\Bigr),
	\end{equation}
	where $B_k=2\bigl(F^\eps-x_k F^\eps_x\bigr)/x_k$ and $F^\eps,F^\eps_x,M_x$ are
	evaluated at $(x_k,t_k)$.

	\emph{Step 5.}
	The cutoff~\eqref{eq:cutoff} satisfies $a\ge0$, $a'\ge0$, $a''\le0$ and
	$a(x)\le x$, so $2a/x-2a'+x a''\le2$. Since $\alpha_k\le0$, multiplying this
	bound by $\eps\alpha_k$ reverses it, so the right-hand side
	of~\eqref{eq:locineq} satisfies
	$\eps\alpha_k\bigl(2a/x-2a'+x a''\bigr)\ge 2\eps\alpha_k$.
	Absorbing
	$2\eps\alpha_k$ into the linear coefficient,
	\begin{equation}\label{eq:locineq2}
		\alpha_k^2+\alpha_k\Bigl(\tfrac32+M_x-2\eps\Bigr)+B_k
		+\rho_k x_k\Bigl(m+\tfrac12+2\eps a'(x_k)-F^\eps_x\Bigr)\ \ge\ 0 .
	\end{equation}

	\emph{Step 6.}
	Since $0\le F^\eps_x\le m\le1$, $a'$ is bounded, and $\rho_k x_k\le1/k$, the
	localization term in~\eqref{eq:locineq2} is $o_k(1)$. So,
	\begin{equation}\label{eq:locineq3}
		\alpha_k^2+\alpha_k\Bigl(\tfrac32+M_x-2\eps\Bigr)+B_k+o_k(1)\ \ge\ 0 .
	\end{equation}

	\emph{Step 7 (the invariant bound).}
	In the $M$-variable $B_k=2(F^\eps-x_k F^\eps_x)/x_k=2\bigl(M_x-M/x_k\bigr)$.
	Concavity of $F^\eps$ with $F^\eps(0,t)=0$ gives $B_k\ge0$, while the
	invariant~\eqref{eq:sec4-inv}
	gives $M/x_k\ge M_x/2$. Thus,
	\begin{equation}\label{eq:newB}
		0\le B_k\le M_x \,.
	\end{equation}
	This single factor of two, giving $B_k\le M_x$ in place of the crude
	$B_k\le2M_x$
	of~\cite{TV}, is what the whole extension turns on.

	\emph{Step 8.}
	By~\eqref{eq:newB} and $\alpha_k\le0$,
	\[
		\alpha_k^2+\alpha_k\Bigl(\tfrac32+M_x-2\eps\Bigr)+B_k
		\ \le\ \alpha_k^2+\alpha_k\Bigl(\tfrac32-2\eps\Bigr)+(1+\alpha_k)M_x .
	\]
	Since $0\le M_x\le m\le1$ and $1+\alpha_k\to0$ while $\alpha_k\to-1$,
	\begin{equation}\label{eq:firstcross}
		\limsup_{k\to\infty}\Bigl(\alpha_k^2+\alpha_k\bigl(\tfrac32+M_x-2\eps\bigr)
		+B_k\Bigr)\ \le\ 1-\Bigl(\tfrac32-2\eps\Bigr)=-\tfrac12+2\eps\ <\ 0
		\qquad(\eps<\tfrac14),
	\end{equation}
	which contradicts~\eqref{eq:locineq3}. Hence $\alpha(T)$ never reaches $-1$.
	Together with~\eqref{eq:sec4-recall}, this gives $-1<xF^\eps_{xx}\le0$ on
	$(0,\infty)^2$.
\end{proof}

Letting $\eps\downarrow0$, the vanishing-viscosity limit
$F=\lim_{\eps\to0}F^\eps$ inherits
\[
	-1\le xF_{xx}\le0
\]
in the same sense, and by the same compactness and stability argument, as in the
proof of Theorem 1.7 of~\cite{TV}. Together with $0\le F_x\le m$, this furnishes
the hypotheses of Theorem~\ref{thm:Fbern}.

\begin{remark}[Discriminant bookkeeping]\label{rem:discriminant}
	The argument uses no discriminant condition. Nonetheless the same
	factor of two
	is visible at the level of the roots of $P(\zeta)=\zeta^2+A\zeta+B$ with
	$A=\tfrac32+M_x-2\eps$. The improved bound $B\le M_x$ together with
	$0\le M_x\le1$ gives
	\[
		A^2-4B\ \ge\ \Bigl(\tfrac32+M_x-2\eps\Bigr)^2-4M_x
		=\Bigl(M_x-\tfrac12-2\eps\Bigr)^2+2-8\eps\ >\ 0\qquad(\eps<\tfrac14),
	\]
	whereas the crude bound $B\le2M_x$ yields a discriminant whose sign
	degenerates
	as $m\uparrow\tfrac12$. This degeneration is the root-level shadow of the
	same loss.
\end{remark}

\subsection{Why $m\le1$ is essential}
The first-crossing bound~\eqref{eq:firstcross} no longer sees $m$.
One might worry that it ``proves too much'', since for $m>1$ there is
no $C^1$ solution
(Theorem 1.4 of~\cite{TV}).
The resolution is that the hypotheses needed even to
\emph{run} the argument fail for $m>1$. The propagation
Proposition~\ref{prop:visc-prop} rests throughout on the gradient
bound~\eqref{eq:Fxbound}: it gives $q=M_x\ge0$ and, in Step~1, both the sign of
the quadratic term and the linear growth $F^{\eps,\delta}\le mx$. In turn,
\eqref{eq:Fxbound} is proved in Lemma~\ref{lem:Fxbound} from the two-sided
bound
$0\le F^{\eps,\delta}\le mx$, where $0$ and $mx$ are respectively a subsolution
and an exact solution of~\eqref{eq:visc-F}, and $0$ is a subsolution precisely
because
\[
	\tfrac12 m(m-1)\le0\quad\Longleftrightarrow\quad m\le1 .
\]
For $m>1$ this breaks
down as $0$ is no longer a subsolution.
So the propagation argument of
Proposition~\ref{prop:visc-prop} no longer applies and the bound $B\le M_x$
is unavailable.

\section{Proof of Theorem~\ref{thm:main}}\label{sec:completion}

We start this section by citing a theorem. Its proof is the existence part of
Section~3 of~\cite{TV}, which is written there for $0<m<\tfrac12$ and applies
on the full range $0<m\le1$ after one constant is adjusted.
Remark~\ref{rem:TVconstant} below identifies the adjustment.

\begin{theorem}
	\label{thm:Fbern}
	Let $0<m\le1$, and let $F_0$ be the Bernstein transform of a nonnegative
	measure $c_0$ on $(0,\infty)$ with
	\[
		m_1(0)=m,\qquad m_0(0)<\infty,\qquad m_2(0)<\infty .
	\]
	Let $F$ be the sublinear viscosity solution of the Hamilton--Jacobi
	equation~\eqref{eq:HJ}, obtained as the vanishing-viscosity limit
	in the scheme
	of~\cite{TV}, and assume that the estimates
	\begin{equation}\label{eq:postbarrier}
		0\le F_x\le m,\qquad -1\le xF_{xx}\le0
	\end{equation}
	hold on $(0,\infty)^2$, in the sense obtained from the vanishing-viscosity
	approximation. Then
	\[
		F\in C^\infty((0,\infty)^2)\cap C^1([0,\infty)^2),\qquad F_x(0,t)=m,
	\]
	and
	\[
		(-1)^{n+1}\,\d_x^n F\ge0\quad\text{on }(0,\infty)^2\quad\text{for
			every }n\ge1 .
	\]
	Consequently, for each $t>0$ the map $x\mapsto F(x,t)$ is a Bernstein function
	and
	\[
		F(x,t)=\int_0^\infty\bigl(1-e^{-sx}\bigr)\,c_t(\dd s)
	\]
	for a nonnegative measure $c_t$ on $(0,\infty)$ with
	\[
		\int_0^\infty s\,c_t(\dd s)=F_x(0,t)=m .
	\]
	The family $(c_t)_{t\ge0}$, with initial datum $c_0$, is a
	mass-conserving weak
	solution of the coagulation--fragmentation
	equation~\eqref{eq:CF} in the measure sense of Definition~\ref{def:measure}.
\end{theorem}

\begin{remark}\label{rem:TVconstant}
	The existence argument in Section~3 of~\cite{TV} consists of their
	Propositions~3.9 and~3.10, Lemmas~3.11 and~3.12, and the proof of their
	Theorem~1.8. It uses the restriction $m<\tfrac12$ in exactly one place.
	Along the characteristics $X(t)$ of their Proposition~3.9 one has
	$\dot X=\d_xF-(m+\tfrac12)$ with $0\le\d_xF\le m$, and their bound (3.19)
	reads $-1\le-(m+\tfrac12)\le\dot X\le-\tfrac12$. The left inequality is
	exactly the statement $m\le\tfrac12$. On the full range $0<m\le1$ it must be
	replaced by
	\[
		-\tfrac32\ \le\ -\Bigl(m+\tfrac12\Bigr)\ \le\ \dot X\ \le\ -\tfrac12 .
	\]
	The estimates in that part of~\cite{TV} which invoke (3.19) through the
	upper bound $\dot X\le-\tfrac12$, such as their integral bound (3.27) and
	the proof of their Lemma~3.11, are unaffected, since that bound does not
	involve $m$. The lower bound enters only through boundedness, namely in the
	finite-time absorption of characteristics at $x=0$ and in the localization
	windows in the proof of their Proposition~3.10, where the constant $-1$
	becomes $-\tfrac32$ and the windows widen accordingly. The arguments use
	only that $\dot X$ is bounded and strictly negative, so nothing structural
	changes. We found no other appeal to $m<\tfrac12$ in that part of~\cite{TV}.
	The remaining occurrences of $\tfrac12$ there are the coefficient
	$m+\tfrac12$ and numerical constants. Finally, the proof of their Lemma~3.11
	uses precisely the bounds $-1\le x\d_x^2F\le0$ of \eqref{eq:postbarrier},
	which Theorem~\ref{thm:barrier} supplies on the full range after
	$\eps\downarrow0$.
\end{remark}

Sections~\ref{sec:initial}--\ref{sec:removal} supply the only new input needed
beyond~\cite{TV} so that the curvature barrier
\[
	0\le F_x\le m,\qquad -1\le xF_{xx}\le0
\]
is now valid on the full critical range $0<m\le1$.

\subsection{Proof of Theorem~\ref{thm:main}}
Existence on $0<m\le1$ follows from the barrier (Theorem~\ref{thm:barrier})
together with Theorem~\ref{thm:Fbern}, which furnishes, for the
Bernstein-transform datum $F_0$ of Theorem~\ref{thm:main}, a nonnegative
measure $c_t$ that is a mass-conserving weak solution of~\eqref{eq:CF}.
Uniqueness on $0<m\le1$ is Corollary 1.3 of~\cite{TV}.
The non-existence of mass-conserving solutions when $m>1$ is
Corollary 1.5 of~\cite{TV}.
\qed

\section{A Keller--Segel curiosity: the same half-slope
  invariant}\label{sec:legendre}

This section is interpretive and is not used in the proof.
The estimate
$2M-xM_x\ge0$ has a counterpart in the
standard radial partial-mass formulation of the two-dimensional
parabolic--elliptic Keller--Segel equation.
Consider the following parabolic--elliptic Keller--Segel equation
\begin{equation}\label{eq:KSsystem}
	\rho_t=\Delta\rho-\nabla\cdot(\rho\nabla c),\qquad -\Delta c=\rho
	\quad\text{in }\R^2 .
\end{equation}
For a radially symmetric solution, the cumulative mass
\begin{equation}\label{eq:Qdef}
	Q(x,t)=\int_{\{|y|^2\le x\}}\rho(y,t)\,\dd y
\end{equation}
in the variable $x=r^2$ satisfies the local equation~\cite{BKLN}
\begin{equation}\label{eq:KSraw}
	Q_t=4x\,Q_{xx}+\frac1\pi\,Q\,Q_x ,
\end{equation}
the diffusion coefficient $4x$ arising from $\d_r^2-\tfrac1r\d_r=4x\,\d_x^2$
under $x=r^2$.
The total mass is $Q(\infty,t)$ and the critical mass is $8\pi$.
Normalizing the mass by $8\pi$ and rescaling time, the normalized
cumulative mass $u=Q/(8\pi)$ satisfies
\begin{equation}\label{eq:KS}
	u_t=xu_{xx}+2uu_x ,
\end{equation}
where now the total mass is $\mu=u(\infty,t)$ and $\mu=1$ is the critical
value.

Let
\[
	P(x,t)=\int_0^x u(s,t)\,\dd s ,\qquad P_x=u .
\]
Integrating \eqref{eq:KS} from $0$ to $x$ (using $u(0,t)=0$) gives
\begin{equation}\label{eq:KSP}
	P_t=xP_{xx}-P_x+P_x^2 .
\end{equation}
The Keller--Segel analogue of $W=2M-xM_x$ is the half-slope
quantity $H:=2P-xP_x$.

For a smooth solution, $H$ satisfies a clean equation. One has
$H_x=P_x-xP_{xx}$ and $H_{xx}=-xP_{xxx}$, while differentiating
\eqref{eq:KSP} gives $(P_t)_x=xP_{xxx}+2P_xP_{xx}$. Hence
\[
	H_t=2P_t-x(P_t)_x
	=-x^2P_{xxx}+2P_x^2-2P_x-2xP_xP_{xx}+2xP_{xx},
\]
and the right-hand side is exactly $xH_{xx}+2(P_x-1)H_x$, so that
\begin{equation}\label{eq:KSH}
	H_t=xH_{xx}+2(P_x-1)H_x .
\end{equation}
Equation \eqref{eq:KSH} is linear in $H$, with drift $2(P_x-1)$ and
nonnegative diffusion coefficient $x$, and $H(0,t)=0$ because $P(0,t)=0$ and
$u(0,t)=0$. One might therefore expect the condition $H(\cdot,0)\ge0$ to be
propagated by the maximum principle. Making this rigorous requires handling
the degeneracy of the diffusion at $x=0$ and imposing decay hypotheses on $u$
at $x=\infty$, which we do not pursue in this aside.

A simple sufficient condition for the initial inequality is concavity of the
partial mass: if $u_0(0)=0$ and $u_0$ is concave, then
$u_0(s)\ge\frac{s}{x}u_0(x)$ for $0\le s\le x$, so
\[
	P_0(x)=\int_0^x u_0(s)\,\dd s\ge\frac{x}{2}u_0(x)=\frac{x}{2}P_{0,x}(x),
\]
that is, $2P_0-xP_{0,x}\ge0$. Since $u_x$ is proportional to the radial density,
this corresponds to a radially nonincreasing density.

The Legendre transform makes the analogy with the coagulation--fragmentation
invariant even more striking. For the computations that follow we assume
$P_{xx}=u_x>0$ so that $y=P_x$
is a genuine change of variables and $G_{yy}=1/P_{xx}$ makes sense.
Likewise, we assume $M_{xx}>0$ on the coagulation--fragmentation side.

Writing $G(y,t)=xy-P(x,t)$ with $y=P_x$, so $G_y=x$ and
$P=xy-G$,
\begin{equation}\label{eq:KShalf}
	2P-xP_x\ge0
	\ \Longleftrightarrow\
	2(xy-G)-xy\ge0
	\ \Longleftrightarrow\
	\frac{G}{G_y}\le\frac{y}{2},
\end{equation}
the same half-slope inequality as the Legendre form of our invariant
$2M-xM_x\ge0$.

The two models do not share the same Legendre
equation, however.
\eqref{eq:KSP} transforms to
\begin{equation}\label{eq:KSleg}
	G_t+\frac{G_y}{G_{yy}}=y(1-y),
\end{equation}
whereas the analogous Legendre transform of $M_t=\tfrac12 M_x(M_x+1)-M/x$ gives
\begin{equation}\label{eq:CFleg}
	G_t+\frac{G}{G_y}=\tfrac12\,y(1-y).
\end{equation}
What they share is the critical polynomial $y(1-y)$ and (potentially) the
propagated half-slope
inequality $G/G_y\le\tfrac{y}{2}$.

Both equations possess a critical mass separating global,
mass-conserving (respectively globally smooth) solutions from gelation
(respectively blow-up). The critical mass is $m=1$ for
coagulation--fragmentation and
$8\pi$ (normalized to $\mu=1$ above) for Keller--Segel.
It would be interesting to investigate whether there is
a deeper relationship between the two models.

\section*{Acknowledgements}
The author thanks Nguyen Van Tien, whose talk at the University of
Sciences in Ho Chi Minh City on June 29th, 2026 inspired the
idea relating the critical coagulation--fragmentation equation to the
Keller--Segel equation.

\medskip
\noindent\textbf{Use of AI tools.}
The large language model Claude (Anthropic) was used to verify
computations, cross-check the dependence on~\cite{TV}, and review
drafts. The review identified gaps in earlier versions of the proofs
of Lemma~\ref{lem:Fxbound} and Proposition~\ref{prop:visc-prop}.
Remark~\ref{rem:TVconstant} was also suggested by Claude to aid
readability. The author implemented and checked all the details and
takes full responsibility for all results and proofs.

\printbibliography

\end{document}